\documentclass[a4paper,12pt]{article}
\usepackage{amsmath}
\usepackage{amssymb}
\usepackage{amsopn}
\usepackage{amstext}
\usepackage{amsthm}
\usepackage{amscd}

\newtheorem{lemma}{Lemma}[section]
\newtheorem{theorem}[lemma]{Theorem}

\newtheorem{remark}[lemma]{Remark}
\newtheorem{corollary}[lemma]{Corollary}

\def\sp{\mathrm{sp}}
\def\M{\mathrm M}

\def\Ran{\mathrm{Ran}\,}
\def\diag{\mathrm{diag\,}}
\def\A{\mathcal A}
\def\B{\mathcal B}

\def\I{\mathrm I}

\def\dcup{\mathop{\bigcup}\limits}

\def\csr{\mathrm{csr\,}}

\begin{document}

\title{Relatively spectral homomorphisms and $K$--injectivity\thanks{Project supported by Natural Science Foundation of China (no.10771069) and Shanghai
Leading Academic Discipline Project(no.B407)}}
\author{Yifeng Xue\thanks{email: yfxue@math.ecnu.edu.cn}\\
Department of Mathematics, East China Normal University\\
Shanghai 200241, P.R.China}
\date{}
\maketitle
\begin{abstract}
Let $\A$ and $\B$ be unital Banach algebras and $\phi\colon\A\rightarrow\B$ be a unital continuous homomorphism.
We prove that if $\phi$ is relatively spectral (i.e., there is a dense subalgebra $X$ of $\A$
such that $\sp_\B(\phi(a))=\sp_\A(a)$ for every $a\in X$) and has dense range, then
$\phi$ induces monomorphisms from $K_i(\A)$ to $K_i(\B)$, $i=0,1$.
\end{abstract}

\baselineskip 17pt

\section{Introduction and Preliminaries}

Let $\A,\,\B$ be  unital Banach algebras and $\phi\colon\A\rightarrow\B$ be a unital
homomorphism (i.e., $\phi(1)=1$). If $\sp_\B(\phi(a))=\sp_\A(a)$ for all $a\in\A$, we say that
$\phi$ is spectral, here $\sp_\A(a)$ denotes the spectrum of $a$ in $\A$. Recall from
\cite[Definition 10]{Nica} that $\phi$ is said to be
relatively spectral (to $X$) if there is a dense subalgebra $X$ of $\A$ such that
$\sp_\B(\phi(x))=\sp_\A(x)$ for all $x\in X$. Furthermore, $\phi$ is said to be completely relatively
spectral if $\phi_n\colon\text{M}_n(\A)\rightarrow \text{M}_n(\B)$ is relatively spectral
(to $\text{M}_n(X)$) for each $n$, where
$\phi_n((a_{ij})_{n\times n})=(\phi(a_{ij}))_{n\times n}$, $(a_{ij})_{n\times n}\in
\text{M}_n(\A)$.

It is known that if $\phi$ is spectral and has dense range, then $\phi$ induces an isomorphism
$K_*(\A)\cong K_*(\B)$ (cf. \cite{Bo}). Recently, B. Nica shows that if $\phi$ is
completely relatively spectral and $\Ran(\phi)=\phi(\A)$ is dense in $\B$, then $\phi$ also induces an
isomorphism $K_*(\A)\cong K_*(\B)$ (cf. \cite[Theorem 2]{Nica}).

Since we do not know if a relatively spectral homomorphism is completely relatively spectral
in general, except some special cases listed in \cite{Nica}, it is significant to investigate if
the relatively spectral homomorphism $\phi$ with dense range could induce an isomorphism between $K_*(\A)$
and $K_*(\B)$.

In this short note, we prove  following result which partially generalizes \cite[Theorem 2]{Nica}.
\begin{theorem}\label{Th}
Let $\A,\,\B$ be two unital Banach algebras and $\phi$ be a unital continuous homomorphism
from $\A$ to $\B$ with dense range. If $\phi$ is relatively spectral, then
$\phi$ induces a monomorphism $K_*(\A)\rightarrow K_*(\B)$.
\end{theorem}

We will prove this theorem in \S 2. Now we introduce some notations used in next section.
For a Banach  algebra $\A$ with unit $1$, let $GL(\A )$ (resp. $GL_0(\A)$) denote the group of
invertible elements in $\A$ (resp. the connected component of $1$ in $GL(\A )$).
For a Banach algebra $\A$, we view $\A^n$ and ${\text M}_n(\A)$ as the set of all $n\times 1$
and $n\times n$ matrices over $\A$ respectively. The norm on $\A^n$ (resp. $\text{M}_n(\A)$)
is given by $\|(a_1,\cdots,a_n)^T\|=\sum\limits^n_{i=1}\|a_i\|$ (resp. $\|(a_{ij})_{n\times n}\|
=\sum\limits^n_{i,j=1}\|a_{ij}\|$). Set $GL_n(\A)=GL({\text
M}_n(\A)),\,GL^0_n(\A)=GL_0({\text M}_n(\A))$.

Suppose $\A$ is unital and let $x\in GL(\A)$. Denote by $[x]$ the equivalence class of $x$ in
$GL(\A)/GL_0(\A)$. The $K_1$--group of $\A$, denoted $K_1(\A)$ is defined as
$K_1(\A)=\bigcup\limits_{n=1}^{\infty}GL_n(\A)/GL_n^0(\A)$, where $GL_n(\A)/GL_n^0(\A)\subset
GL_{n+1}(\A)/GL_{n+1}^0(\A)$ in the sense that $[x]\mapsto[\diag(x,1)]$, $\forall\,x\in
GL_n(\A)/GL_n^0(\A)$, $n=1,2,\cdots$. We can define the $K_0$--group of $\A$ by
$K_0(\A)=K_1((S\A)^+)$, where
$$
(S\A)^+=\{f\in C([0,1],\A)\vert\,f(0)=f(1)=\text{constant}\}.
$$
More detailed information about $K_0(\A)$ and $K_1(\A)$ can be found in \cite{Bl}.

\section{Proof of main theorem}

In this section, we assume that $\A$ and $\B$ are unital Banach algebras and
$\phi$ is a  unital continuous homomorphism from $\A$ to $\B$.

Let $M$ be a compact Hausdorff space and let $C(M,\A)$ denote the Banach algebra consisting of
all continuous maps $f\colon M\rightarrow\A$ with the norm $\|f\|=
\sup\limits_{t\in M}\|f(t)\|$. When $M=\mathbf S^1$, we set $\Omega(\A)=C(\mathbf S^1,\A)$.
Let $X$ be a dense subalgebra of $\A$ and put $M(X)=\{f\colon M\rightarrow X\
\text{continuous}\}$. Define a homomorphism $\phi_M\colon C(M,\A)\rightarrow C(M,\B)$ by
$\phi_M(f)(t)=\phi(f(t))$, $\forall\,f\in C(M,\A)$ and $t\in M$.

\begin{lemma}\label{la}
Let $\phi$ be a relatively spectral (to $X$) homomorphism with $\Ran(\phi)$ dense in $\B$.
Then $\phi_M$ is a relatively spectral (to $M(X)$) homomorphism with $\Ran(\phi_M)$ dense in
$C(M,\B)$.
\end{lemma}
\begin{proof}Given $g\in C(M,\B)$ and $\epsilon>0$. Since $g$ is continuous and $M$ is compact,
it follows that $g(M)$ is also compact in $\B$. Thus, there are $y_1,\cdots,y_n\in g(M)$
such that $g(M)\subset\dcup_{i=1}^nO(y_i,\epsilon)$, where $O(y_i,\epsilon)=\{y\in\B\vert\,
\|y-y_i\|<\epsilon\}$. Set $U_i=g^{-1}(O(y_i,\epsilon))$, $i=1,\cdots,n$. Then $\{U_1,\cdots,U_n\}$
is an open cover of $M$. Choose a partition of unity $\{f_1,\cdots,f_n\}$ subordinate
to this cover, so that each $f_i$ is a continuous function from $M$ to $[0,1]$ with
support contained in $U_i$ and $\sum\limits^n_{i=1}f_i(t)=1$, $\forall\,t\in M$.

From $\overline{\phi(X)}=\B$, we can find $a_1,\cdots,a_n\in X$ such that
$\|\phi(a_i)-y_i\|<\epsilon$, $i=1\,\cdots,n$. Set $f_\epsilon(t)=\sum\limits^n_{i=1}a_if_i(t)$, $\forall\,
t\in M$. Then $f\in M(X)$ and
$$
\|\phi(f_\epsilon(t))-g(t)\|\le\sum\limits^n_{i=1}\|\phi(a_i)-y_i\|f_i(t)+\sum\limits^n_{i=1}
\|y_if_i(t)-g(t)f_i(t)\|<2\epsilon,
$$
$\forall\,t\in M$. Thus, $\phi_M(M(X))$ is dense in $C(M,\B)$.

If we set $\B=\A$ and $\phi=\mathrm{id}$ in above argument, then we have $\overline{M(X)}
=C(M,\A)$.

Now we show that $\phi_{M}$ is relatively spectral. But it is enough to prove that
$\phi_M(f)$ is invertible in $C(M,\B)$ for $f\in M(X)$ implies that $f$ is
invertible in $C(M,\A)$. Since $\phi_M(f)$ is invertible in $C(M,\B)$, it follows that
$\phi(f(t))$ is invertible in $\B$, $\forall\,t\in M$. Thus, from the relatively spectral
property of $\phi$, we have $f(t)\in GL(\A)$, $\forall\,t\in M$. This means that $f\in GL(C(M,\A))$.
\end{proof}

\begin{corollary}\label{ca}
Let $\phi$ be a relatively spectral (to $X$) homomorphism with dense range. Suppose there is
$a\in X$ such that $\|1-\phi(a)\|<1$. Then $a\in GL_0(\A)$.
\end{corollary}
\begin{proof}Choose $x\in X$ such that $\|1-x\|\le\dfrac{1}{1+\|\phi\|}$.
Then $x\in GL_0(\A)$ and $\|1-\phi(x)\|<1$. Put $f(t)=(1-t)x+ta$, $\forall\,t\in\I=[0,1]$.
Then $f\in\I(X)$ and $\|1-\phi_I(f)\|<1$. So $f\in GL(C(I,\A))$ with $f_0=x$ and $f_1=a$
by Lemma \ref{la}, which means that $a\in GL_0(\A)$.
\end{proof}

\begin{lemma}\label{lb}
Let $\phi$ be a relatively spectral ( to $X$) homomorphism with dense range and let
$z\in\M_n(X)$ with $\|1_n-\phi_n(z)\|<\dfrac{1}{\,3\,}$, where $1_n$ is the unit of $\M_n(\A)$.
Then for any $\epsilon>0$, there is $z'\in\M_n(X)\cap GL^0_n(\A)$ such that $\|z-z'\|<\epsilon$.
\end{lemma}
\begin{proof}
When $n=1$, the statement is true by Corollary \ref{ca}. We assume that the statement is true
for $1\le n\le m$. We now prove the argument is also true for $n=m+1$.

Let $y=\phi_{m+1}(z)\in\phi_{m+1}(\M_{m+1}(X))$ with $\|1_{m+1}-y\|<\dfrac{1}{\,3\,}$. Write
$z=(z_{ij})_{m+1\times m+1}$ (resp. $y=(y_{ij})_{m+1\times m+1}$) as $z=\begin{pmatrix}z_1&z_2\\
z_3&z_4\end{pmatrix}$ (resp. $y=\begin{pmatrix}y_1&y_2\\ y_3&y_4\end{pmatrix}$),
where $y_{ij}=\phi(z_{ij})$, $i,j=1,\cdots,m+1$, $z_1=(z_{ij})_{m\times m}\in\M_m(X)$ and
$$
z_2=\begin{pmatrix}z_{1\,m+1}\\ \vdots\\ z_{m\,m+1}\end{pmatrix},\quad
z_3=\begin{pmatrix}z_{m+1\,1}&\dots & z_{m+1\,m}\end{pmatrix},\quad z_4=z_{m+1\,m+1}.
$$
Then $\|1_m-y_1\|<\dfrac{1}{\,3\,}$, $\|y_2\|<\dfrac{1}{\,3\,}$, $\|y_3\|<\dfrac{1}{\,3\,}$
and $\|1-y_4\|<\dfrac{1}{\,3\,}$. By assumption, there is $z_1'\in GL_m^0(\A)\cap\M_m(X)$
such that $\|z_1-z_1'\|<\dfrac{\epsilon}{3(\|\phi_m\|+1)}$. Pick $x_1\in\M_m(X)$ such that
\begin{equation}\label{eq1}
\|(z_1')^{-1}-x_1\|<\dfrac{\epsilon}{3(\|z_1'\|+1)(\|z_3\|+1)(\|\phi_m\|+1)}.
\end{equation}
Set
$$
z'=\begin{pmatrix}1_n&\ 0\\ z_3x_1&\ 1\end{pmatrix}
\begin{pmatrix}z_1'&z_2\\ \ 0&\ z_4-z_3x_1z_2\end{pmatrix}
=\begin{pmatrix}z_1'&z_2\\ z_3x_1z_1'&z_4\end{pmatrix}\in\M_{m+1}(X).
$$
Then $\|z-z'\|\le\|z-z'\|+\|z_3\|\|1_m-x_1z_1'\|<\epsilon$.

Note that $\|1_m-\phi_m(z_1)\|<\dfrac{1}{\,3\,}$ and $\|\phi_m(z_1)-\phi_m(z_1')\|
<\dfrac{1}{\,3\,}$. So $\|1_m-\phi_m(z_1')\|<\dfrac{2}{\,3\,}$ and hence $\|(\phi_m(z_1'))^{-1}\|
<3$. Consequently, $\|\phi_m(x_1)\|<4$ by (\ref{eq1}). Therefore,
\begin{equation}\label{eq2}
\|1-\phi(z_4-z_3x_1z_2)\|<\|1-y_4\|+\|y_3\|\|\phi_m(x_1)\|\|y_2\|<1.
\end{equation}
Applying Corollary \ref{ca} to (\ref{eq2}), we have $z_4-z_3x_1z_2\in GL_0(\A)$. Thus,
we can deduce that $\begin{pmatrix}z_1'&z_2\\ \ 0&\ z_4-z_3x_1z_2\end{pmatrix}\in GL^0_{m+1}(\A)$.
Since $\begin{pmatrix}1_n&\ 0\\ z_3x_1&\ 1\end{pmatrix}\in GL_{m+1}^0(\A)$, it follows that
$z'\in GL_{m+1}^0(\A)$. This completes the proof.
\end{proof}

Now we give the proof of Theorem \ref{Th} as follows.

\begin{proof}Let $G\in GL_n(\A)$ such that $G_0=\phi_n(G)\in GL_n^0(\B)$ for some $n$. We will
prove $G\in GL_n^0(\A)$. Since $X$ is dense in $\A$ and $\phi$ has dense range, we can find
$A\in\M_n(X)$ with $\|A-G\|$ small enough so that $A\in GL_n(\A)$ with $[A]=[G]$ in
$GL_n(\A)/GL_n^0(\A)$ and $\phi_n(\A)\in GL_n^0(\B)$.
Noting that $\phi_n(A)$ can be written as $\phi_n(A)=e^{b_1}\cdots e^{b_s}$ for some
$b_1,\cdots,b_s\in\text M_n(\B)$,
we can find $a_1,\cdots,a_s\in\text M_n(\A)$ with $\|\phi_n(a_i)-b_i\|$ small enough, $i=1,\cdots,s$, such that
$$
\|\phi_n(A)^{-1}-\phi_n(e^{-a_1}\cdots e^{-a_s})\|<\dfrac{1}{6(\|\phi_n(A)\|+1)}.
$$
Choose $B_0\in\M_n(X)$ such that
$$
\|e^{-a_1}\cdots e^{-a_s}-B_0\|<\dfrac{1}{6(\|e^{a_1}\cdots e^{a_s}\|+1)
(\|\phi_n\|+1)(\|\phi_n(A)\|+1)}.
$$
Then
$B_0\in GL_n^0(\A)$ and
$$
\|I_n-\phi_n(AB_0)\|\le\|\phi_n(A)\|\|\phi_n(A)^{-1}-\phi_n(B_0)\|<\frac{1}{\,3\,}.
$$
Therefore there exists $Z\in GL_n^0(\A)\cap\M_n(X)$ such that $\|AB_0-Z\|<\dfrac{1}{\|(AB_0)^{-1}\|}$
by Lemma \ref{lb}. So, $AB_0\in GL_n^0(\A)$ and hence $G\in GL_n^0(\A)$.

Since $\phi_{\mathbf S^1}$ is relatively spectral and $\Ran(\phi_{\mathbf S^1})$ is dense in
$\Omega(\B)$
by Lemma \ref{la}, we get that the induced
homomorphism $(\phi_{\mathbf S^1})_*\colon K_1(\Omega(\A))\rightarrow K_1(\Omega(\B))$ is
injective by above argument. Thus, from the commutative diagram of split exact sequences
\begin{equation}\label{xyz}
\begin{CD}
0@> >>K_1(S\B)@> >>K_1(\Omega(\B))@> >>K_1(\B)@> >>0 \\
@. @A\phi_*AA @A(\phi_{\mathbf S^1})_*AA @A\phi_* AA\\
0@> >>K_1(S\A)@> >>K_1(\Omega(\A))@> >>K_1(\A)@> >>0,
\end{CD}
\end{equation}
we get that $\phi_*\colon K_1(S\A)\rightarrow K_1(S\B)$ is injective, that is,
$\phi_*\colon K_0(\A)\rightarrow K_0(\B)$ is injective.
\end{proof}
\begin{remark}
{\rm
If $\csr(C(\bf S^1,\B))\le 2$, where $\csr(\cdot)$ is the connected stable rank of Banach
algebras introduced by Rieffel in \cite{R}, then $\phi_*$ is also surjective.

In fact, when $\csr(C(\mathbf S^1,\B))\le 2$, we have $\csr(\B)\le 2$ by
\cite[Proposition 8.4]{N2}. So the natural homomorphism $i_\B\colon GL(\B)/GL_0(\B)\rightarrow
K_1(\B)$ is surjective by using the proof of Theorem 10.10 in \cite{R}. On the other hand,
if $\phi$ is relatively spectral and has dense range, then the induced homomorphism
$\phi_*\colon GL(\A)/GL_0(\A)\rightarrow GL(\B)/GL_0(\B)$ is injective by the proof of Theorem
\ref{Th}. To show $\phi_*$ is surjective, let $b\in GL(\B)$ and pick $a\in X$ such that
$\|\phi(a)-b\|<\dfrac{1}{\|b^{-1}\|} $. Then $\phi(a)\in GL(\B)$ with $[b]=[\phi(a)]$ and hence
$a\in GL(\A)$. Thus, $\phi_*([a])=[b]$ and consequently, $\phi_*\colon K_1(\A)\rightarrow K_1(\B)$ is surjective.

By the above argument, we also have $(\phi_{\mathbf S^1})_*$ is surjective. Thus, using the
commutative diagram (\ref{xyz}), we obtain that $\phi_*\colon K_1(S\A)\rightarrow K_1(S\B)$
is surjective.

Especially, when $\B$ is of topological stable rank one, $\csr(\B)\le 2$ and
$\csr(C(\mathbf S^1,\B))\le 2$. So $\phi_*$ is surjective.
}
\end{remark}
\vspace{3mm}

\noindent{\bf{Acknowledgement.\,}}The author is grateful to the referee for his (or her)
very helpful comments and suggestions.


\vspace{2mm}

\begin{thebibliography}{19}

\bibitem{Bl}B. Blackadar, ``$K$--Theory for Operator Algebras'', Springer--Verlag, 1986.
\bibitem{Bo}J.-B. Bost, \emph{Principe d'Oka, $K$--th\'{e}orie et syst\'{e}mes
dynamiques non commutatifs}, Invent. Math., {\bf 101}(2) (1990), 261--333.

\bibitem{Nica}B. Nica, \emph{Relatively spectral morphisms and applications to $K$--theory}, J. Funct. Anal., 255 (2008),
3303--3328.
\bibitem{N2}B. Nica, \emph{Homotopical stable ranks for Banach algebras}, http://arxiv.org/
abs/0911.2945v1.

\bibitem{R}M.A. Rieffel, \emph{Dimension and stable rank in the $K$--Theory of $C^*$--algebras}, Proc. London Math. Soc.,
{\bf 46}(1983), 301--333.

\end{thebibliography}
\end{document}